\documentclass[a4paper,10pt,leqno]{amsart}

\RequirePackage[T1]{fontenc}
\RequirePackage[utf8]{inputenc}
\RequirePackage{xcolor}
	\definecolor{lapis}{HTML}{26619C}
\RequirePackage{amsfonts,amssymb,amsmath,amsrefs}
\RequirePackage{comment,etoolbox,mathtools}

\RequirePackage{hyperref}
	\hypersetup{colorlinks,allcolors=lapis}
\RequirePackage{lmodern,stmaryrd,textcomp}

\usepackage[capitalize]{cleveref}
	
	\calclayout
	
\numberwithin{equation}{section}
\theoremstyle{plain}
	\newtheorem{thm}{Theorem}
\theoremstyle{plain}
	\newtheorem{conj}[thm]{Conjecture}
	\newtheorem{cor}[thm]{Corollary}
	\newtheorem{lem}[thm]{Lemma}
	\newtheorem{prop}[thm]{Proposition}
\theoremstyle{definition}

\theoremstyle{remark}
	\newtheorem{rem}[thm]{Remark}

	\Crefname{thm}{Theorem}{Theorems}
	\Crefname{lem}{Lemma}{Lemmas}
	
\usepackage{autonum}
	
\newcommand{\bb}{\mathbf}
\newcommand{\cal}{\mathcal}

\newcommand{\ur}{\mathrm}

\renewcommand{\sf}{\mathsf}

\newcommand{\define}[3]{\expandafter#1\csname#3\endcsname{#2{#3}}}
\forcsvlist{\define{\DeclareMathOperator}{}}{ad,codim,coker,gr,id,im,ord,rk,sgn,supp,tr,val}
\forcsvlist{\define{\DeclareMathOperator}{}}{Ad,Aut,End,Ext,Gal,Hom,Ind,Inf,Irr,Iso,Lie,Mat,Res,Spec,Sym,Tor}
\forcsvlist{\define{\DeclareMathOperator}{}}{GL,PGL,SL,Sp}
\forcsvlist{\define{\newcommand}{\mathrm}}{ab,nil,op,red,reg,rs,sep,tor}

\newcommand{\pa}[1]{\left(#1\right)}
\newcommand{\mat}[1]{\begin{matrix}#1\end{matrix}}
\newcommand{\pmat}[1]{\pa{\mat{#1}}}

\RequirePackage{tikz-cd}
\RequirePackage[cal=rsfso]{mathalfa}

\DeclareMathOperator{\Alt}{\text{$\Lambda$}}

\newcommand{\Br}{\mathit{Br}}
\newcommand{\QL}{\bar{\mathbf{Q}}_\ell}

\makeatletter
\g@addto@macro \normalsize {%
 \setlength\abovedisplayskip{10pt plus 2pt minus 2pt}%
 \setlength\belowdisplayskip{10pt plus 2pt minus 2pt}%
}
\makeatother

\begin{document}

\title{Unipotent Elements and Twisting in Link Homology}
\author{Minh-T\^{a}m Quang Trinh}
\address{Massachusetts Institute of Technology, 77 Massachusetts Avenue, Cambridge, MA 02139}
\email{mqt@mit.edu}

\maketitle

\begin{abstract}
Let $\mathcal{U}$ be the unipotent variety of a complex reductive group $G$.
Fix opposed Borel subgroups $B_\pm \subseteq G$ with unipotent radicals $U_\pm$.
The map that sends $x_+x_- \mapsto x_+x_-x_+^{-1}$ for all $x_\pm \in U_\pm$ restricts to a map from $U_+U_- \cap gB_+$ into $\mathcal{U} \cap gB_+$, for any $g$.
We conjecture that the restricted map forms half of a homotopy equivalence between these varieties, and thus, induces a weight-preserving isomorphism between their compactly-supported cohomologies.
Noting that the map is equivariant with respect to certain actions of $B_+ \cap gB_+g^{-1}$, we prove for type $A$ that an equivariant analogue of this isomorphism exists.
Curiously, this follows from a certain duality in Khovanov--Rozansky homology, a tool from knot theory.
\end{abstract}

\thispagestyle{empty}



\section{Introduction}

\subsection{}

Let $G$ be a complex, connected, reductive algebraic group.
Let $B_+$ and $B_-$ be opposed Borel subgroups of $G$, and let $U_\pm$ be the unipotent radical of $B_\pm$.
For instance, if $G$ is the group of invertible $n \times n$ matrices $\GL_n$, then we can choose $U_+$, \emph{resp.}\ $U_-$, to be the subgroup of upper-triangular, \emph{resp.}\ lower-triangular, matrices with $1$'s along the diagonal.
Every element of $U_+U_-$ can be written uniquely in the form $x_+x_-$, where $x_\pm \in U_\pm$.

Let $\cal{U} \subseteq G$ be the closed subvariety of unipotent elements.
With the notation above, there is a map $\Phi : U_+U_- \to \cal{U}$ defined by
\begin{align}
\Phi(x_+x_-) = x_+x_-x_+^{-1}.
\end{align}
Note that $U_+U_-$ is an affine space, whereas $\cal{U}$ is usually singular.
Nonetheless, in the analytic topology, the sets $U_+U_-(\bb{C})$ and $\cal{U}(\bb{C})$ are homotopy equivalent, as they are both contractible.

For any $g \in G(\bb{C})$, the map $\Phi$ restricts to a map $\Phi_g : \cal{V}_g \to \cal{U}_g$, where
\begin{align}
\cal{U}_g &= \cal{U} \cap gB_+,\\
\cal{V}_g &= U_+U_- \cap gB_+.
\end{align}
Note that $\cal{U}_g, \cal{V}_g, \Phi_g$ only depend on the coset $gB_+$.
We endow $\cal{U}_g(\bb{C})$ and $\cal{V}_g(\bb{C})$ with the analytic topology.
The new idea proposed in this note is:

\begin{conj}\label{conj:main}
The map $\Phi_g$ defines half of a homotopy equivalence between $\cal{U}_g(\bb{C})$ and $\cal{V}_g(\bb{C})$.
\end{conj}

The simplest case is $g \in B_+(\bb{C})$.
Here, $\cal{U}_g(\bb{C}) = U_+(\bb{C}) = \cal{V}_g(\bb{C})$ and $\Phi_g$ is a retract from the affine space $U_+(\bb{C})$ onto the point corresponding to the identity of $G(\bb{C})$.
In general, neither $\cal{U}_g(\bb{C})$ nor $\cal{V}_g(\bb{C})$ will be contractible, as we will show in \Cref{sec:low-rank} through examples.

We emphasize that for an arbitrary map of varieties $\Phi : Y \to X$ that induces a homotopy equivalence on $\bb{C}$-points, there may not exist a nonconstant map of varieties $X \to S$ such that $\Phi$ induces a homotopy equivalence of fibers $Y_s(\bb{C}) \to X_s(\bb{C})$ for every $s \in S(\bb{C})$.
For instance, take $\Phi$ to be the projection from a quadric cone onto its axis of symmetry.
Both total spaces are contractible, so $\Phi$ is automatically a homotopy equivalence.
For the map of fibers over $s$ to be a homotopy equivalence as well, $X_s(\bb{C})$ must contract onto the origin of $X(\bb{C}) = \bb{C}$.
Since $X \to S$ is nonconstant, some $s$ must violate this condition.

\subsection{}

Some motivation for \Cref{conj:main} comes from classical results about finite groups of Lie type.
To state them, let us implicitly replace $G$ with its split form over a finite field $\bb{F}$ of good characteristic.

In \cite[Thm.\ 15.1]{steinberg}, Steinberg showed the identity $|\cal{U}(\bb{F})| = |U_+U_-(\bb{F})|$.
In \cite[\S{4}]{kawanaka}, Kawanaka showed an identity equivalent to
\begin{align}\label{eq:kawanaka}
|\cal{U}_g(\bb{F})| = |\cal{V}_g(\bb{F})|;
\end{align}
see \Cref{rem:kawanaka}.
More recently, Lusztig has given a new proof of \eqref{eq:kawanaka} in \cite{lusztig}.

\Cref{conj:main} essentially implies \eqref{eq:kawanaka}.
Indeed, once one checks that $\cal{U}_g$ and $\cal{V}_g$ have the same dimension, \Cref{conj:main} implies that $\Phi_g$ induces an isomorphism between the rational, compactly-supported cohomologies of $\cal{U}_g$ and $\cal{V}_g$.
Since $\Phi_g$ is algebraic, this isomorphism must match their weight filtrations in the sense of mixed Hodge theory.
One can further check that both sides of \eqref{eq:kawanaka} are polynomial functions of $|\bb{F}|$, so by the results explained in \cite{katz}, the virtual weight polynomials of $\cal{U}_g$ and $\cal{V}_g$ specialize to their $\bb{F}$-point counts.

Our main result is evidence for an equivariant analogue of \Cref{conj:main}.
Let $T = B_+ \cap B_-$, so that $B_+ = TU_+ \simeq T \ltimes U_+$.
The map $\Phi$ transports the $B_+$-action on $U_+U_-$ defined by
\begin{align}\label{eq:action}
tu \cdot x_+x_- = (tux_+t^{-1})(tx_-t^{-1})
\end{align}
for all $(t, u) \in T \times U_+$ onto the $B_+$-action on $\cal{U}$ by left conjugation.
Setting
\begin{align}
H_g = B_+ \cap gB_+g^{-1},
\end{align}
we find that these $B_+$-actions restrict to $H_g$-actions on $\cal{V}_g$ and $\cal{U}_g$, respectively.
The $B_+$-equivariance of $\Phi$ thus restricts to $H_g$-equivariance of $\Phi_g$.
Though we do not have a general theorem about $\Phi_g$ itself, we prove:

\begin{thm}\label{thm:cohomology}
If $G$ is a split reductive group of type $A$ over $\bb{F}$, then for all $g \in G(\bb{F})$, there is an isomorphism of bigraded vector spaces
\begin{align}
\gr_\ast^\sf{W} \ur{H}_{c, H_g}^\ast(\cal{U}_g, \QL)
\simeq
\gr_\ast^\sf{W} \ur{H}_{c, H_g}^\ast(\cal{V}_g, \QL),
\end{align}
where $\ur{H}_{c, H_g}^\ast(-, \QL)$ denotes $H_g$-equivariant, compactly-supported $\ell$-adic cohomology and $\sf{W}$ denotes its weight filtration.
\end{thm}

The proof in \Crefrange{sec:flag}{sec:kr} uses ideas from the rather different world of low-dimensional topology, as we now explain.

Let $\cal{X}_g$ be the variety of Borel subgroups of $G$ in generic position with respect to both $B_+$ and $gB_+g^{-1}$.
In \Cref{sec:flag}, we will construct an isomorphism $\cal{X}_g \to \cal{V}_g$ that transports the $H_g$-action on $\cal{X}_g$ by left conjugation to the $H_g$-action on $\cal{V}_g$ described above.

The variety $\cal{X}_g$ is closely related to the so-called braid varieties that have been studied recently by several authors, including \cite{mellit, sw, cggs, gl}.
To be more precise:
Let $W$ be the Weyl group of $G$, and let $\Br_W^+$ be the positive braid monoid of $W$.
By definition, $\Br_W^+$ is generated by elements $\sigma_w$ for each $w \in W$ modulo
\begin{align}
\sigma_{ww'} &= \sigma_w\sigma_{w'}
	&\text{whenever $\ell(ww') = \ell(w) + \ell(w')$},
\end{align}
where $\ell$ is the Bruhat length function on $W$.
A braid variety is a configuration space of tuples of Borels, where the relative positions of cyclically consecutive Borels have constraints determined by a fixed element $\beta \in \Br_W^+$.
Note that there is a central element $\pi \in \Br_W^+$ known as the full twist and given by $\pi = \sigma_{w_0}^2$, where $w_0$ is the longest element of $W$.
In the notation of \cite[Appendix B]{trinh}, the variety $\cal{X}_g$ is isomorphic to the braid variety attached to $\sigma_w\pi$, where $w$ is the relative position of the pair $(B_+, gB_+g^{-1})$.

In \cite{trinh}, it was shown that the weight filtration on the equivariant, compactly-supported cohomology of the braid variety of $\beta$ encodes a certain summand of a certain triply-graded vector space attached to $\beta$, known as its \emph{HOMFLYPT} or \emph{Khovanov--Rozansky (KR) homology}.
When $W$ is the symmetric group $S_n$, the braid $\beta$ represents the isotopy class of a topological braid on $n$ strands, and the KR homology of $\beta$ is an isotopy invariant of the link closure of $\beta$, up to grading shifts \cite{kr, khovanov}.

At the same time, when $w$ is the relative position of $(B_+, gB_+g^{-1})$, it turns out that $\cal{U}_g$ is closely related to another variety attached to $\sigma_w$ in \cite{trinh}.
Just as $\cal{X}_g$ encodes the ``highest $a$-degree'' of the KR homology of $\sigma_w \pi$, so $\cal{U}_g$ encodes the ``lowest $a$-degree'' of the KR homology of $\sigma_w$.
For $W = S_n$, Gorsky--Hogancamp--Mellit--Nagakane have established an isomorphism between these bigraded vector spaces for general braids $\beta$, not just $\sigma_w$, which they deduce from an analogue of Serre duality for the homotopy category of Soergel bimodules for $S_n$ \cite{ghmn}.
Our \Cref{thm:cohomology} follows from this isomorphism.

This proof suggests that we regard \Cref{conj:main} as a \emph{geometric} realization of the (purely algebraic) Serre duality of \cite{ghmn}.
Moreover, it suggests an extension of \cref{conj:main} with positive braids in place of elements of $W$.
We state the extended conjecture in \Cref{sec:kr}.

The isomorphism of \emph{ibid.}\ categorifies an earlier identity of K\'alm\'an, relating the bivariate HOMFLY series of the link closures of $\beta$ and $\beta\pi$.
In \cite{trinh}, we generalized K\'alm\'an's theorem to the braids associated with arbitrary finite Coxeter groups.
(We expect, but do not show, that \cite{ghmn} admits a similar generalization.)
For Weyl groups, we also interpreted our result as an identity of $\bb{F}$-point counts.
In \Cref{sec:ih}, we review these results, and explain how they recover \eqref{eq:kawanaka}.

\subsection{Acknowledgments}

I thank George Lusztig and Zhiwei Yun for helpful discussions.
During this work, I was supported by a Mathematical Sciences Postdoctoral Fellowship (grant \#{2002238}) from the National Science Foundation.

\section{Reductions}

In this section, we collect lemmas about easy reductions and special cases of the conjecture.

\begin{lem}\label{lem:prod}
If $G = G_1 \times G_2$, where $G_1, G_2$ are again reductive, then \Cref{conj:main} holds if and only if it holds with $G_1$ in place of $G$ and with $G_2$ in place of $G$.
\end{lem}

\begin{proof}
We can factor $U_\pm = U_{\pm, 1} \times U_{\pm, 2}$ and $\cal{U} = \cal{U}_1 \times \cal{U}_2$ and $\Phi = \Phi^{(1)} \times \Phi^{(2)}$, where $U_{\pm, i}, \cal{U}_i, \Phi^{(i)}$ are the analogues of $U_\pm, \cal{U}, \Phi$ with $G_i$ in place of $G$.
\end{proof}

\begin{lem}\label{lem:der}
Both $\cal{U}_g$ and $\cal{V}_g$ are contained within the derived subgroup $G^\ur{der} \subseteq G$.
In particular, we can assume $G$ is semisimple in \Cref{conj:main}.
\end{lem}

\begin{proof}
It is enough to show $\cal{U} \subseteq G^\ur{der}$, because in that case, we also have $U_+U_- \subseteq G^\ur{der}G^\ur{der} = G^\ur{der}$.
Suppose instead that $\cal{U} \not\subseteq G^\ur{der}$.
Then $G/G^\ur{der}$ contains nontrivial unipotent elements, because homomorphisms of algebraic groups preserve Jordan decompositions \cite[\S{9.21}]{milne}.
But $G/G^\ur{der}$ is isogeneous to the center of $G$ by \cite[Ex.\ 19.25]{milne}, so it is a torus, in which the only unipotent element is the identity.
\end{proof}

\begin{lem}\label{lem:ad}
The map $\Phi_g$ is unchanged, up to composition with maps that induce homeomorphisms on $\bb{C}$-points, when we replace $G$ with the adjoint group $G^{\ad} = G/Z(G)$ and $g$ with its image in $G^{\ad}$.
In particular, \Cref{conj:main} is equivalent to its analogue where we replace $G$ with any quotient by a central isogeny, and $g$ with its image in that quotient.
\end{lem}

\begin{proof}
Let $\bar{U}_\pm$ and $\bar{\cal{U}}$ be the respective analogues of $U_\pm$ and $\cal{U}$ with $G^{\ad}$ in place of $G$.
Again using the preservation of Jordan decomposition, and the fact that central elements of $G$ are semisimple, we can check that the quotient map $G \to G^{\ad}$ restricts to maps $U^\pm \xrightarrow{\sim} \bar{U}^\pm$ and $\cal{U} \to \bar{\cal{U}}$ that give rise to bijections on field-valued points.
Writing $\bar{\Phi} : \bar{U}_+\bar{U}_- \to \bar{\cal{U}}$ for the analogue of $\Phi$, we see that the diagram
\begin{equation}
\begin{tikzpicture}[baseline=(current bounding box.center), >=stealth]
\matrix(m)[matrix of math nodes, row sep=2.5em, column sep=2.5em, text height=2ex, text depth=0.5ex]
{ 		
	U_+U_-
		&\cal{U}\\	
	\bar{U}_+\bar{U}_-
		&\bar{\cal{U}}\\
};
\path[->, font=\scriptsize, auto]
(m-1-1)
	edge node{$\Phi$} (m-1-2)
	edge node[below,sloped]{$\approx$} (m-2-1)
(m-1-2)
	edge node[sloped]{$\approx$} (m-2-2)
(m-2-1)
	edge node{$\bar{\Phi}$} (m-2-2);
\end{tikzpicture}
\end{equation}
commutes.
Now the result follows.
\end{proof}

The following lemma is motivated by the Bruhat decomposition
\begin{align}
G = \coprod_{w \in W} U_+\dot{w}B_+,
\end{align}
where $w \mapsto \dot{w}$ is any choice of set-theoretic lift from $W \simeq N_G(T)/T$ into $N_G(T)$.
Note that $\dot{w}B_+$ only depends on $w$ because $T \subseteq B_+$.

\begin{lem}\label{lem:translate}
If \Cref{conj:main} holds for some $g \in G(\bb{C})$, then it holds with $ug$ in place of $g$, for all $u \in U_+(\bb{C})$.
In particular, if \Cref{conj:main} holds for all $g \in N_G(T)(\bb{C})$, then it holds in general.
\end{lem}

\begin{proof}
We can factor $\Phi_{ug}$ as a composition
\begin{align}
\cal{V}_{ug} \xrightarrow{\sim} \cal{V}_g \xrightarrow{\Phi_g} \cal{U}_g \xrightarrow{\sim} \cal{U}_{ug},
\end{align}
where the first arrow is left multiplication by $u^{-1}$, and the last arrow is left conjugation by $u$.
Since these are both isomorphisms of varieties, we get the first assertion of the lemma.
The second follows from the first via Bruhat.
\end{proof}

If $P_+$ and $P_-$ are opposed parabolic subgroups of $G$ containing $B_+$ and $B_-$, respectively, and $L = P_+ \cap P_-$ is their common Levi subgroup, then we write $U_{\pm, P}$ for the unipotent radical of $P_\pm$, and write $B_{\pm, L}, U_{\pm, L}, \cal{U}_L$ for the analogues of $B_\pm, U_\pm, \cal{U}$ with $L$ in place of $G$.
Thus we have
\begin{align}
P_\pm &= LU_{\pm, P} \simeq L \ltimes U_{\pm, P},\\
B_\pm &= B_{\pm, L}U_{\pm, P} \simeq B_{\pm, L} \ltimes U_{\pm, P},\\
U_\pm &= U_{\pm, L}U_{\pm, P} \simeq U_{\pm, L} \ltimes U_{\pm, P}.
\end{align}
If $g \in L$, then we write $\cal{U}_{L, g}, \cal{V}_{L, g}, \Phi_{L, g}$ for the analogues of $\cal{U}_g, \cal{V}_g, \Phi_g$ with $L$ in place of $G$.

\begin{lem}\label{lem:levi}
Let $P_\pm \supseteq B_\pm$ be opposed parabolic subgroups of $G$, and let $L = P_+ \cap P_-$.
Then for all $g \in L(\bb{C})$, we have isomorphisms of algebraic varieties
\begin{align}
\cal{U}_g &\simeq \cal{U}_{L, g}U_{+, P} \simeq \cal{U}_{L, g} \times U_{+, P},\\
\cal{V}_g &\simeq \cal{V}_{L, g}U_{+, P} \simeq \cal{V}_{L, g} \times U_{+, P}.
\end{align}
In particular, if $g \in G(\bb{C})$ belongs to a Levi subgroup of $G$, then in \Cref{conj:main}, we can replace $G$ with that Levi subgroup.
\end{lem}

\begin{proof}
Since the decomposition $P_+ \simeq L \ltimes U_{+, P}$ preserves Jordan decompositions, we have $\cal{U} \cap P_+ = \cal{U}_L U_{+, P} \simeq \cal{U}_L \times U_{+, P}$.
Intersecting with $gB_+$, we get
\begin{align}
\cal{U} \cap gB_+ = (\cal{U}_L \cap gB_{+, L})U_{+, P} \simeq (\cal{U}_L \cap gB_{+, L}) \times U_{+, P}.
\end{align}
Next, $U_+U_- \cap gB_+ \subseteq gB_+ \subseteq P_+$ and $U_- \cap P_+ = U_{-, L} \cap P_+$ together imply $U_+U_- \cap gB_+ = U_+U_{-, L} \cap gB_+$, from which
\begin{align}
U_+U_- \cap gB_+
&= U_{+, L} U_{+, P} U_{-, L} \cap gB_{+, L}U_{+, P}\\
&= (U_{+, L}U_{-, L} \cap gB_{+, L})U_{+, P}\\
&\simeq (U_{+, L}U_{-, L} \cap gB_{+, L}) \times U_{+, P}.
\end{align}
So it remains to prove the last assertion in the lemma.
For this, it is more convenient to use the decomposition $\cal{V}_g = U_{+, P}\cal{V}_{L, g} \simeq U_{+, P} \times \cal{V}_{L, g}$.
For all $(x, y_+, y_-) \in U_{+, P} \times U_{+, L} \times U_{-, L}$, observe that 
\begin{align}
\Phi_g(xy_+y_-) = xy_+y_-y_+^{-1}x^{-1} = \Ad_x(\Phi_{L, g}(y_+y_-)),
\end{align}
where $\Ad_x(u) = xux^{-1}$.
A choice of deformation retract from $U_{+, P}(\bb{C})$ onto $\{1\}$ induces a homotopy from $\Ad_x : \cal{U}_g \to \cal{U}_g$ onto $\id : \cal{U}_g \to \cal{U}_g$, which in turn induces a homotopy from $\Phi_g : \cal{V}_g \to \cal{U}_g$ onto the composition
\begin{align}
\cal{V}_g \xrightarrow{p_{L, g}} \cal{V}_{L, g} \xrightarrow{\Phi_{L, g}} \cal{U}_{L, g} \xrightarrow{i_{L, g}} \cal{U}_g,
\end{align}
where $p_{L, g} :\cal{V}_g \to \cal{V}_{L, g}$ is the retract induced by the projection $U_{+, P}(\bb{C}) \to \{1\}$ and $i_{L, g} : \cal{U}_{L, g} \to \cal{U}_g$ is the section induced by the inclusion $\{1\} \to U_{+, P}(\bb{C})$.
Therefore, $\Phi_{L, g}$ being half of a homotopy equivalence is equivalent to $\Phi_g$ being half of a homotopy equivalence.
\end{proof}

\begin{lem}\label{lem:id}
\Cref{conj:main} holds for $g \in B_+(\bb{C})$.
\end{lem}

As observed in the introduction, this can be proved by computing $\cal{U}_g(\bb{C}), \cal{V}_g(\bb{C})$, and $\Phi_g$ directly.
Alternatively:

\begin{proof}
Since $\cal{U}_g, \cal{V}_g, \Phi_g$ only depend on $gB_+$, we can assume $g = 1$.
Then \Cref{lem:levi} reduces us to the case $G = T$, where $\cal{U}_g = \{1\} = \cal{V}_g$.
\end{proof}

\section{Examples}\label{sec:low-rank}

\subsection{}

By \Crefrange{lem:prod}{lem:ad}, it suffices to check \Cref{conj:main} for one representative $G$ from each central isogeny class of semisimple algebraic group with connected Dynkin diagram.
Moreover, by \Cref{lem:translate}, it suffices to fix a lift $w \mapsto \dot{w}$ from $W$ into $N_G(T)$ and check \Cref{conj:main} for cosets of the form $gB_+ = \dot{w}B_+$.
In this section, we settle the conjecture completely for $G \in \{\SL_2, \SL_3\}$, and check one further case in which $G = \Sp_4$.

Without loss of generality, we can always take $U_+$, \emph{resp.}\ $U_-$, to be the group of unipotent upper-triangular, \emph{resp.}\ unipotent lower-triangular, matrices in $G$.
To produce defining equations for $\cal{U}$, we use the coefficients of the map that sends $g \in G$ to its characteristic polynomial.
We write $e$ for the identity element of $W$.

\subsection{The Group $\mathrm{SL}_2$}

We have
\begin{align}
\cal{U} = \left\{g \in \SL_2 \mid
	\tr(g) = 2
\right\}.
\end{align}
The map $\Phi : U_+U_- \to \cal{U}$ is
\begin{align}
\Phi\pa{\pmat{1&b\\ &1}\pmat{1&\\ b'&1}}
=	\pmat{1 + bb'&-b^2b'\\ b'&1 - bb'}.
\end{align}
We can write $W = \{e, w_0\}$.
By \Cref{lem:translate,lem:id}, it suffices to check $gB_+ = \dot{w}_0B_+$.
The varieties $\cal{U}_g$ and $\cal{V}_g$ are
\begin{align}
\cal{U}_g
&=	\left\{\pmat{ &-\frac{1}{X} \\ X &2} \,\middle|\, X \neq 0\right\},\\[1ex]
\cal{V}_g
&=	\left\{\pmat{ &-\frac{1}{b'} \\ b' &1} \,\middle|\, b' \neq 0\right\}.
\end{align}
The coordinates define isomorphisms $b' : \cal{U}_g \xrightarrow{\sim} \bb{G}_m$ and $X : \cal{V}_g \xrightarrow{\sim} \bb{G}_m$.
The map $\Phi_g$ is an isomorphism of varieties, corresponding to setting $X = b'$.

\subsection{The Group $\mathrm{SL}_3$}

Let $\Lambda^2(g)$ denote the exterior square of a matrix $g$.
From the identity $2\tr(\Lambda^2(g)) = \tr(g)^2 - \tr(g^2)$, we have
\begin{align}
\cal{U} = \left\{g \in \SL_3 \,\middle|
	\begin{array}{l}
	\tr(g) = 3,\\
	\tr(\Lambda^2(g)) = 3
	\end{array}
	\right\}
= \left\{g \in \SL_3 \,\middle|
	\begin{array}{l}
	\tr(g) = 3,\\
	\tr(g^2) = 3
	\end{array}
	\right\}.
\end{align}
The map $\Phi : U_+U_- \to \cal{U}$ is
\begin{align}
&\Phi\pa{\pmat{1&a&b\\ &1&c\\ &&1}\pmat{1&&\\ a'&1&\\ b'&c'&1}}\\[1ex]
&=	\pmat{1 + aa' + bb'
			&bc' - a^2a' -abb'
			&-aba' - b^2b' - bcc' + a^2ca' + abcb'\\
		a' + cb'
			&1 - aa' + cc' - acb'
			&-ba' + aca' - bcb' - c^2c' + ac^2b'\\
		b'
			&c' - ab'
			&1 - bb' - cc' + acb'}.
\end{align}
We can write $W = \{e, s, t, ts, st, w_0\}$, where $s$ and $t$ are the simple reflections.
The simple reflections lift to elements of $N_G(T)$ contained in proper Levi subgroups of $G$, so by the $\SL_2$ case and \Cref{lem:der,lem:translate,lem:levi}, it remains to consider $gB_+ = \dot{w}B_+$ for $w \in \{ts, st, w_0\}$.
In what follows, we choose $\dot{s}, \dot{t}$ so that
\begin{align}
\dot{s}B_+ \subseteq \left\{\pmat{ &\ast &\ast \\ \ast &\ast &\ast \\ & &\ast}\right\}
\quad\text{and}\quad
\dot{t}B_+ \subseteq \left\{\pmat{ \ast &\ast &\ast \\ & &\ast \\ &\ast &\ast }\right\}.
\end{align}

\subsubsection{}

If $w = ts$, then the varieties $\cal{U}_g$ and $\cal{V}_g$ are
\begin{align}
\cal{U}_g
	&= \left\{\pmat{
		&Y &C \\ 
		& &Z \\ 
		\frac{1}{YZ} &-\frac{1}{Z}(3 + \frac{C}{YZ}) &3
	}\,\middle|\,
		Y, Z \neq 0
\right\},\\[1ex]
\cal{V}_g
	&= \left\{\pmat{ 
		&-\frac{1}{a'} &b \\ 
		& &c \\ 
		-\frac{a'}{c} &-\frac{1}{c} &1
	}\,\middle|\,
		a', c \neq 0
\right\}.
\end{align}
The coordinates define isomorphisms $(Y, Z, C) : \cal{U}_g \xrightarrow{\sim} \bb{G}_m^2 \times \bb{A}^1$ and $(c, a', b) : \cal{V}_g \xrightarrow{\sim} \bb{G}_m^2 \times \bb{A}^1$.
The map
\begin{align}
\Phi_g(x_+x_-)
&= \pmat{ &-\frac{1}{a'} &b + \frac{c}{a'} \\ & &c \\ -\frac{a'}{c} &-\frac{1}{c}(2 - \frac{ba'}{c}) &3}
\end{align}
is an isomorphism of varieties, corresponding to $(Y, Z, C) = (-\frac{1}{a'}, c, b + \frac{c}{a'})$.

\subsubsection{}

If $w = st$, then the varieties are
\begin{align}
\cal{U}_g
	&= \left\{\pmat{
		& &\frac{1}{XY} \\ 
		X &A &-\frac{1}{Y}(3 - 3A + A^2) \\ 
		&Y &3 - A
	}\,\middle|\,
		X, Y \neq 0
\right\},\\[1ex]
\cal{V}_g
	&= \left\{\pmat{ 
		& &b \\ 
		a' &1 + \frac{c}{ba'} &c \\ 
		&\frac{1}{ba'} &1
	}\,\middle|\,
		b, a' \neq 0
\right\}.
\end{align}
The coordinates define isomorphisms $(X, Y, A) : \cal{U}_g \xrightarrow{\sim} \bb{G}_m^2 \times \bb{A}^1$ and $(b, a', c) : \cal{V}_g \xrightarrow{\sim} \bb{G}_m^2 \times \bb{A}^1$.
The map
\begin{align}
\Phi_g(x_+x_-)
&= \pmat{ & &b \\ a' &2 + \frac{c}{ba'} &-(ba' + \frac{c^2}{ba'} + c) \\ &\frac{1}{ba'} &1 - \frac{c}{ba'}}
\end{align}
is an isomorphism of varieties, corresponding to $(X, Y, A) = (a', \frac{1}{ba'}, 1 + \frac{c}{ba'})$.

\subsubsection{}

If $w = w_0$, then the varieties are
\begin{align}
\cal{U}_g 
	&= \left\{\pmat{ 
		& &Z \\ 
		&-\frac{1}{XZ} &C \\ 
		X &A &3 + \frac{1}{XZ}
	}\,\middle|
		\begin{array}{l}
			X, Z \neq 0,\\
			(1 + \frac{1}{XZ})^3 + \frac{AC}{XZ} = 0
		\end{array}
\right\},\\[1ex]
\cal{V}_g 
	&= \left\{\pmat{ 
		&&b \\ 
		&1 + c'c &c \\ 
		b'&c'&1
	}\,\middle|
		\begin{array}{l}
			b, b' \neq 0,\\
			1 + bb' + (bb')(cc') = 0
		\end{array}
\right\}.
\end{align}
The coordinates define isomorphisms
\begin{align}
\cal{U}_g
	&\xrightarrow{\sim} 
	\{(X, Z, A, C) \in \bb{G}_m^2 \times \bb{A}^2 \mid (1 + \tfrac{1}{XZ})^3 + \tfrac{AC}{XZ} = 0\},\\
\cal{V}_g
	&\xrightarrow{\sim} 
	\{(b, b', c, c') \in \bb{G}_m^2 \times \bb{A}^2 \mid 1 + bb' + (bb')(cc') = 0\}.
\end{align}
The map 
\begin{align}
\Phi_g(x_+x_-)
&=	
	\pmat{&&b\\ &1+cc'&(1 + \frac{1}{bb'})c\\ b'&(1 + bb')c'&2-cc'}
\end{align}
corresponds to setting $(X, Z, A, C) = (b', b, (1 + bb')c', (1 + \frac{1}{bb'})c)$.
Note that $\cal{U}_g$ and $\cal{V}_g$ are \emph{not} isomorphic as varieties.

\begin{prop}
For $G = \SL_3$ and $gB_+ = \dot{w}_0B_+$, the map $\Phi_g$ is neither injective nor surjective on $\bb{C}$-points, but does define half of a homotopy equivalence.
\end{prop}

\begin{proof}
Let 
\begin{align}
\cal{U}_g^\dagger 
	&= \{(u, A, C) \in \bb{G}_m \times \bb{A}^2 \mid AC = -(1 + u)(1 + \tfrac{1}{u})^2\},\\
\cal{V}_g^\dagger 
	&= \{(u, c, c') \in \bb{G}_m \times \bb{A}^2  \mid cc' = -(1 + \tfrac{1}{u})\}.
\end{align}
Let $\Phi_g^\dagger : \cal{V}_g^\dagger \to \cal{U}_g^\dagger$ be the map $\Phi_g^\dagger(u, c, c') = (u, (1 + u)c', (1 + \frac{1}{u})c)$.
Then $\Phi_g$ is a pullback of $\Phi_g^\dagger$, so it suffices to show the claim of the proposition with $\cal{U}_g^\dagger, \cal{V}_g^\dagger, \Phi_g^\dagger$ in place of $\cal{U}_g, \cal{V}_g, \Phi_g$.

Observe that $\Phi_g^\dagger$ preserves $u \in \bb{G}_m$.
Over the subvariety of $\bb{G}_m$ where $u \neq -1$, the fibers of $\cal{U}_g^\dagger$ and $\cal{V}_g^\dagger$ are copies of $\bb{G}_m$: say, via the coordinates $A$ and $c'$.
In these coordinates, $\Phi_g^\dagger$ amounts to rotating $\bb{G}_m$ by $1 + u$.
Over the point $u = -1$, the fibers are copies of the transverse intersection of two lines.
Altogether, $\cal{U}_g^\dagger(\bb{C})$ and $\cal{V}_g^\dagger(\bb{C})$ are both homotopic to pinched tori, and $\Phi_g^\dagger$ induces a self-map of the pinched torus that preserves its longitude and top homology.
Thus $\Phi_g^\dagger$ fits into a homotopy equivalence.
It is neither injective nor surjective because $\Phi_g^\dagger(-1, c, c') = (-1, 0, 0)$.
\end{proof}

\subsection{The Group $\mathrm{Sp}_4$}

We set $\Sp_4 = \{g \in \GL_4 \mid g^tJg = J\}$, where
\begin{align}
J = \pmat{&&&1 \\ &&1 \\ &-1\\ -1}.
\end{align}
For $\Sp_4$, the only nontrivial coefficients of the characteristic polynomial are $\tr(g) = \tr(\Alt^3(g))$ and $\tr(\Alt^2(g))$, so similarly to $\SL_3$, we have:
\begin{align}
\cal{U} = \left\{g \in \Sp_4 \,\middle|
	\begin{array}{l}
	\tr(g) = 4,\\
	\tr(\Lambda^2(g)) = 4
	\end{array}
	\right\}
= \left\{g \in \Sp_4 \,\middle| 
	\begin{array}{l}
	\tr(g) = 4,\\
	\tr(g^2) = 4
	\end{array}
	\right\}.
\end{align}
The map $\Phi : U_+U_- \to \cal{U}$ is
\begin{align}
&\Phi\pa{\pmat{1&a&b+ad&c\\ &1&2d&b-ad\\ &&1&-a\\ &&&1}\pmat{1&&&\\ a'&1&&\\ b'+a'd'&2d'&1&\\ c'&b'-a'd'&-a'&1}}\\[1ex]
&=	\pmat{
	1 + f_1 + g_1
		&f_{1,2} + g_{1,2}
		&f_{1,3} + g_{1,3}
		&h_{1,4}\\
	f_{2,1} + g_{2,1}
		&1 - f_1 + g_2
		&h_{2,3}
		&f_{1,3} - g_{1,3}\\
	f_{3,1} + g_{3,1}
		&h_{3,2}
		&1 - f_1 - g_2
		&f_{1,2} - g_{1,2}\\
	h_{4,1}
		&f_{3,1} - g_{3,1}
		&f_{2,1} - g_{2,1}
		&1 + f_1 - g_1
},
\end{align}
where we set
\begin{align}
f_1	&= ba'd' + ada'd',\\
g_1	&= aa' + bb' + cc' + adb',\\
g_2 &= -aa' + bb' + 4dd' - abc' - 3adb' + a^2dc'
\end{align}
and
\begin{align}
f_{1,2} &= -(c + ab + a^2d)a'd',\\
g_{1,2} &= 2bd' + cb' - a(aa' + bb' + cc' - 2dd' + adb'),\\
f_{1,3} &= -ca' - b(aa' + bb' + cc' + 4dd') - 2cdb' + ad(aa' + cc' - 4dd' + adb'),\\
g_{1,3} &= -(b^2 - 2cd - a^2d^2)a'd',\\ 
f_{2,1} &= 2da'd',\\
g_{2,1} &= a' + bc' + 2db' - adc',\\
f_{3,1} &= b' - ac',\\
g_{3,1} &= a'd'
\end{align}
and
\begin{align}
h_{1,4} &= -c(2aa' + 2bb' + cc') - 2b^2d' - 2ad(cb' + 2bd' + add'),\\
h_{2,3} &= -2ba' + 2d(aa' - 2bb' - 4dd') - b^2c' + ad(2bc' + 4db' - adc'),\\
h_{3,2} &= 2d' - 2ab' + a^2c',\\
h_{4,1} &= c'.
\end{align}
We can write $W = \{e, s, t, ts, st, sts, tst, w_0\}$, where $s$ and $t$ are the simple reflections.
By the $\SL_2$ case and \Cref{lem:der,lem:translate,lem:levi}, it remains to consider $gB_+ = \dot{w}B_+$ for $w \in \{ts, st, sts, tst, w_0\}$.

Below, we will only check $w = sts$.
Without loss of generality, we can assume
\begin{align}
\dot{sts}B_+ 
	= \left\{\pmat{ 
		& & &\frac{1}{X} \\ 
		&Y &2YD &Y(B - AD) \\ & &\frac{1}{Y} &-\frac{A}{Y} \\ -X &-XA &-X(B + AD) &C}
			\,\middle|\,
		X, Y \neq 0
		\right\}.
\end{align}
The varieties $\cal{U}_g$ and $\cal{V}_g$ are
\begin{align}
\cal{U}_g
	&= \small\left\{\pmat{
 		& & &\frac{1}{X}\\
 		&Y &2YD &Y(B - AD)\\
 		& &\frac{1}{Y} &-\frac{A}{Y}\\
 		-X &-XA &-X(B + AD) &4 - Y - \frac{1}{Y}
	}\,\middle|
		\begin{array}{l}
			X, Y \neq 0,\\
			XA(Y(B - AD) - \frac{1}{Y}(B + AD))\\
			{} = \frac{1}{Y^2}(1 - Y)^4
		\end{array}
\right\},\\[1ex]
\cal{V}_g
	&= \small\left\{\pmat{
		{}
			&{}
			&{}
			&c\\[1ex]
		{}
			&\tfrac{1}{1 + aa'}
			&2(1 + aa')d - c(a')^2
			&ca' - 2ad\\[1ex]
		{}
			&{}
			&1 + aa'
			&-a\\[1ex]
		-\frac{1}{c}
			&-\frac{a}{c(1 + aa')}
			&-a'
			&1
	}\,\middle|
	\begin{array}{l}
	c, 1 + aa' \neq 0
	\end{array}
	\right\}.
\end{align}
The coordinates define isomorphisms
\begin{align}
\cal{U}_g
	&\xrightarrow{\sim} 
	\left\{(X, Y, A, B, D) \in \bb{G}_m^2 \times \bb{A}^3 \,\middle|
		\begin{array}{l}
		XA(Y(B - AD) - \tfrac{1}{Y}(B + AD))\\
		{} = \tfrac{1}{Y^2}(1 - Y)^4
		\end{array}
	\right\},\\[1ex]
\cal{V}_g
	&\xrightarrow{\sim} 
	\{(c, a, d, a') \in \bb{G}_m \times \bb{A}^3 \mid 1 + aa' \neq 0\}.
\end{align}
The map
\begin{align}
\Phi_g(x_+x_-)
&=	\pmat{
		{}
			&{} 
			&{}
			&c\\[1ex]
		{}
			&\frac{1}{1 + aa'}
			&2ada'(\frac{2 + aa'}{1 + aa'}) - c(a')^2
			&2a^2 da' - \frac{a^2c(a')^3}{1 + aa'}\\[1ex]
		{} 
			&{} 
			&1 + aa'
			&a^2a'\\[1ex]
		-\frac{1}{c} 
			&\frac{a^2a'}{c(1 + aa')}
			&-\frac{2a^2da'}{c(1 + aa')}
			&3 - aa' - \frac{1}{1 + aa'}
}
\end{align}
corresponds to setting
\begin{align}
\pmat{X \\ Y \\ A \\ B \\ D}
&= \pmat{
		\tfrac{1}{c}\\[1ex]
		\tfrac{1}{1 + aa'}\\[1ex]
		-\tfrac{a^2a'}{1 + aa'}\\[1ex]
		a^2da'(1 + aa' + \tfrac{1}{1 + aa'}) - \frac{1}{2}a^2 c (a')^3\\[1ex]
		ada'(2 + aa') - \tfrac{1}{2} c(a')^2(1 + aa')
	}.
\end{align}

\begin{prop}
For $G = \Sp_4$ and $gB_+ = \dot{sts}B_+$, the map $\Phi_g$ is neither injective nor surjective on $\bb{C}$-points, but does define half of a homotopy equivalence.
\end{prop}

\begin{proof}
The map $\Phi_g$ fits into a commutative diagram
\begin{equation}
\begin{tikzpicture}[baseline=(current bounding box.center), >=stealth]
\matrix(m)[matrix of math nodes, row sep=2.5em, column sep=2.5em, text height=2ex, text depth=0.5ex]
{ 		
	\cal{V}_g
		&\cal{U}_g\\	
	\cal{V}_g^\dagger
		&\cal{U}_g^\dagger\\
};
\path[->, font=\scriptsize, auto]
(m-1-1)
	edge node{$\Phi_g$} (m-1-2)
	edge node[left]{$\Psi_\cal{V}$} (m-2-1)
(m-1-2)
	edge node{$\Psi_\cal{U}$} (m-2-2)
(m-2-1)
	edge node{$\Phi_g^\dagger$} (m-2-2);
\end{tikzpicture}
\end{equation}
where the new varieties are
\begin{align}
\cal{U}_g^\dagger
	&= \{(X, Y, A, B_-, B_+) \in \bb{G}_m^2 \times \bb{A}^3
			\mid
				XA(YB_- - \tfrac{B_+}{Y}) 
				= \tfrac{1}{Y^2}(1 - Y)^4
		\},\\
\cal{V}_g^\dagger
	&= \{(c, u, a_1, a_2, d') \in \bb{G}_m^2 \times \bb{A}^3
			\mid
				(u - 1)^4 = a_1a_2
		\},
\end{align}
and the new maps are
\begin{align}
\Phi_g^\dagger(c, u, a_1, a_2, d')
	&= (\tfrac{1}{c},\quad \tfrac{1}{u},\quad -\tfrac{a_1}{u},\quad 2ud' - ca_2,\quad \tfrac{2d'}{u}),\\
\Psi_\cal{U}(X, Y, A, B, D)
	&= (X,\quad Y,\quad A,\quad B - AD,\quad B + AD),\\
\Psi_\cal{V}(c, a, d, a')
	&= (c,\quad 1 + aa',\quad a^2a',\quad a^2(a')^3, \quad a^2da').
\end{align}
The map $\Phi_g^\dagger$ is algebraically invertible, so to show that $\Phi_g$ induces a homotopy equivalence, it remains to study the topology of the maps $\Psi_\cal{U}$ and $\Psi_\cal{V}$.

To show that $\Psi_\cal{U}$ induces a homotopy equivalence, we first note that it preserves $(X, A) \in \bb{A}^2$.
Over the subvariety of $\bb{A}^2$ where $A \neq 0$, it is invertible.
Over the line $A = 0$, the defining equations of $\cal{U}_g$ and $\cal{U}_g^\dagger$ both simplify to $(1 - Y)^4 = 0$, which has the unique solution $Y = 1$ over $\bb{C}$, so over this line, the fibers of $\cal{U}_g(\bb{C})$ and $\cal{U}_g^\dagger(\bb{C})$ are contractible, being copies of $\bb{C}^2$.

To show that $\Psi_\cal{V}$ induces a homotopy equivalence, it suffices to show the same for the map from $\{(a, a') \in \bb{A}^2 \mid aa' \neq -1\}$ into $\{(u, a_1, a_2) \in \bb{G}_m \times \bb{A}^2 \mid (u - 1)^4 = a_1a_2\}$ that sends $(a, a') \mapsto (1 + aa', a^2a', a^2(a')^3)$.
Indeed, this map restricts to an isomorphism from the subvariety where $aa' \neq 0$ onto the subvariety where $a_1a_2 \neq 0$, and collapses the subvariety where $aa' = 0$ onto the point $(u, a_1, a_2) = (1, 0, 0)$.
The set of $\bb{C}$-points in the domain where $aa' = 0$ is contractible, and the set of $\bb{C}$-points in the target where $a_1a_2 = 0$ admits a deformation retract onto $\{(1, 0, 0)\}$.

Therefore, $\Phi_g$ induces a a homotopy equivalence on $\bb{C}$-points.
It is not injective because $\Phi_g(c, a, d, 0) = (\tfrac{1}{c}, 1, 0, 0, 0) = \Phi_g(c, 0, d, a')$ for all $(c, d, a, a')$ such that $aa' \neq -1$, and it is not surjective because the points of the form $(X, 1, 0, B, D) \in \cal{V}_g(\bb{C})$ with $B \neq 0$ do not appear in the image.
\end{proof}

\section{Configurations of Flags}\label{sec:flag}

\subsection{}

In this section, we relate $\cal{U}_g$ and $\cal{V}_g$ to varieties that were studied in \cite{trinh}.
Henceforth, we fix any field $\bb{F}$ of good characteristic for $G$, and replace $G$ with its split form over $\bb{F}$.
We also assume that $B_+$ is defined over $\bb{F}$.

Let $\cal{B}$ be the flag variety of $G$, \emph{i.e.}, the variety that parametrizes its Borel subgroups.
As these subgroups are all self-normalizing and conjugate to one another, there is an isomorphism of varieties:
\begin{align}\label{eq:coset}
\begin{array}{rcl}
G/B_+ &\xrightarrow{\sim} &\cal{B}\\
xB_+ &\mapsto &xB_+x^{-1}
\end{array}
\end{align}
It transports the $G$-action on $G/B_+$ by left multiplication to the $G$-action on $\cal{B}$ by left conjugation.

The orbits of the diagonal $G$-action on $\cal{B} \times \cal{B}$ can be indexed by the Weyl group $W$.
The closure order on the orbits corresponds to the Bruhat order on $W$ induced by the Coxeter presentation.
For all $(B_1, B_2) \in \cal{B} \times \cal{B}$ and $w \in W$, we write $B_1 \xrightarrow{w} B_2$ to indicate that $(B_1, B_2)$ belongs to the $w$th orbit, in which case we say that it is in \emph{relative position} $w$.
In particular, note that $B_+ \xrightarrow{w_0} B_-$ because $B_- = \dot{w}_0 B_+ \dot{w}_0^{-1}$.

Under the Bruhat decomposition, 
\eqref{eq:coset} restricts to an isomorphism
\begin{align}\label{eq:bruhat}
U_+\dot{w}B_+/B_+ \xrightarrow{\sim} \{B \in \cal{B} \mid B_+ \xrightarrow{w} B\}.
\end{align}
Each side is isomorphic to an affine space of dimension $\ell(w)$, where $\ell : W \to \bb{Z}_{\geq 0}$ is the Bruhat length function.

\subsection{}

Fix $g \in G(\bb{F})$.
Recall that $H_g = B_+ \cap gB_+g^{-1}$ acts on $\cal{V}_g = U_+U_- \cap gB_+$ according to \eqref{eq:action}.
Let
\begin{align}
\cal{X}_g = \{B \in \cal{B} : B_+ \xrightarrow{w_0} B \xrightarrow{w_0} gB_+g^{-1}\},
\end{align}
and let $H_g$ act on $\cal{X}_g$ by left conjugation.
We will prove:

\begin{prop}\label{prop:v-to-x}
There is a $H_g$-equivariant isomorphism of varieties $\cal{X}_g \to \cal{V}_g$.
\end{prop}

We give the proof in two steps.
For convenience, we set $\cal{Y}_g = gU_+\dot{w}_0B_+/B_+ \subseteq G/B_+$.
Let $H_g$ act on $\cal{Y}_1 \cap \cal{Y}_g$ by left multiplication.

\begin{lem}
The isomorphism \eqref{eq:bruhat} for $w = w_0$ restricts to an $H_g$-equivariant isomorphism $\cal{Y}_1 \cap \cal{Y}_g \xrightarrow{\sim} \cal{X}_g$.
\end{lem}

\begin{proof}
Recall that \eqref{eq:bruhat} is $G$-equivariant.
Under the action of an element $x$, the $w = w_0$ case is transported to an isomorphism
\begin{align}
\cal{Y}_x \xrightarrow{\sim} \{B \in \cal{B} \mid xB_+x^{-1} \xrightarrow{w_0} B\}.
\end{align}
On the right-hand side, the direction of the arrow $\xrightarrow{w_0}$ can be reversed because $w_0^{-1} = w_0$.
Now take the fiber product of the isomorphisms for $x = 1$ and $x = g$ over the isomorphism \eqref{eq:coset}.
\end{proof}

In what follows, recall that via the decomposition $B_+ \simeq U_+ \rtimes T$, any element of $B_+$ can be written as $ut$ for some uniquely determined $(u, t) \in U_+ \times T$.
As a consequence, we also get a decomposition $\dot{w}_0 B_+ = U_-\dot{w}_0T = U_-T\dot{w}_0$.

\begin{lem}\label{lem:v-to-y}
The map
\begin{align}
\begin{array}{rcl}
\cal{V}_g = U_+U_- \cap gB_+ &\to &\cal{Y}_1 \cap \cal{Y}_g\\
x_+x_- = gut &\mapsto &x_+x_-t^{-1}\dot{w}_0B_+ = gu\dot{w}_0B_+
\end{array}
\end{align}
is an $H_g$-equivariant isomorphism of varieties.
\end{lem}

\begin{proof}
Let $\cal{V}'_g = U_+\dot{w}_0B_+ \cap gU_+\dot{w}_0$.
Since the map
\begin{align}
\begin{array}{rcl}
\cal{V}_g &\to &\cal{V}'_g\\
x_+x_- = gut &\mapsto &x_+x_-t^{-1}\dot{w}_0 = gu\dot{w}_0
\end{array}
\end{align}
is an isomorphism, it remains to show that the map $f : \cal{V}'_g \to \cal{Y}_1 \cap \cal{Y}_g$ given by
\begin{align}
\cal{V}'_g \to \cal{V}'_g B_+ \to (\cal{V}'_gB_+)/B_+ = \cal{Y}_1 \cap \cal{Y}_g
\end{align}
is bijective on $R$-points for every $\bb{F}$-algebra $R$.
For convenience, we suppress $R$ in the notation below.

Let $yB_+ \in (\cal{V}'_gB_+)/B_+$.
Then we can write $y = u\dot{w}_0b = gu'\dot{w}_0b'$ for some $u, u' \in U_+$ and $b, b' \in B_+$.
Therefore, $yB_+ = f(y(b')^{-1})$, where $y(b')^{-1} = gu'\dot{w}_0 \in \cal{V}'_g$.
This proves $f^{-1}(yB_+)$ is nonempty.
We claim that $f^{-1}(yB_+)$ contains only one point.
Recall that the map $U_+ \to U_+\dot{w}_0B_+/B_+$ that sends $v \mapsto v\dot{w}_0B_+$ is an isomorphism.
Thus, $v \neq u$ implies $v\dot{w}_0B_+ \neq u\dot{w}_0B_+$.
We deduce that
\begin{align}
f^{-1}(yB_+) 
&= f^{-1}(u\dot{w}_0B_+)\\
&\subseteq u\dot{w}_0B_+ \cap gU_+\dot{w}_0\\ 
&\simeq (\dot{w}_0^{-1} g^{-1} u\dot{w}_0)B_+ \cap U_-.
\end{align}
But the intersection of $U_-$ with any coset of $B_+$ contains only one point.
\end{proof}

\subsection{}

Let $\cal{O}_w, \cal{U}_w, \cal{X}_w$ be the varieties defined by
\begin{align}
\cal{O}_w &= \{(B', B'') \in \cal{B} \times \cal{B} \mid B' \xrightarrow{w} B''\},\\
\cal{U}_w &= \{(u, B') \in \cal{U} \times \cal{B} \mid B' \xrightarrow{w} uB'u^{-1}\},\\
\cal{X}_w &= \{(B, B', B'') \in \cal{B} \times \cal{B} \times \cal{B} \mid B' \xrightarrow{w_0} B \xrightarrow{w_0} B'' \xleftarrow{w} B'\}.
\end{align}
Let $G$ act on these varieties by (diagonal) left conjugation.
We regard $\cal{U}_w$ and $\cal{X}_w$ as varieties over $\cal{O}_w$ via the $G$-equivariant maps $(u, B') \mapsto (B', uB'u^{-1})$ and $(B, B', B'') \mapsto (B', B'')$, respectively.

Let $H_g$ act on $G$ by right multiplication.
For any variety $X$ with an $H_g$-action, let $H_g$ act diagonally on $X \times G$, and let $G$ act on $(X \times G)/H_g$ by left multiplication on the second factor.
Finally, fix a prime $\ell > 0$ invertible in $\bb{F}$, so that we can form the equivariant $\ell$-adic compactly-supported cohomology groups
\begin{align}
\ur{H}_{c, H_g}^\ast(X) \simeq \ur{H}_{c, G}^\ast((X \times G)/H_g).
\end{align}
With these conventions, we have:

\begin{prop}\label{prop:g-to-w}
If $B_+ \xrightarrow{w} gB_+g^{-1}$, then there are $G$-equivariant isomorphisms
\begin{align}
(\cal{U}_g \times G)/H_g \xrightarrow{\sim} \cal{U}_w,\\
(\cal{X}_g \times G)/H_g \xrightarrow{\sim} \cal{X}_w.
\end{align}
In particular, they induce isomorphisms on compactly-supported cohomology:
\begin{align}
\ur{H}_{c, H_g}^\ast(\cal{U}_g, \QL) \xrightarrow{\sim} \ur{H}_{c, G}^\ast(\cal{U}_w, \QL),\\
\ur{H}_{c, H_g}^\ast(\cal{X}_g, \QL) \xrightarrow{\sim} \ur{H}_{c, G}^\ast(\cal{X}_w, \QL).
\end{align}
\end{prop}

\begin{proof}
The maps $(\cal{U}_g \times G)/H_g \to \cal{U}_w$ and $(\cal{X}_g \times G)/H_g \to \cal{X}_w$ are
\begin{align}
{[u, x]} &\mapsto (xux^{-1}, xB_+x^{-1}),\\
{[B, x]} &\mapsto (xBx^{-1}, xB_+x^{-1}, xgB_+g^{-1}x^{-1}),
\end{align}
respectively.
To show that they are isomorphisms:
Observe that $G$ acts transitively on $\cal{O}_w$, and the stabilizer of $(B_+, gB_+g^{-1})$ is precisely $H_g$.
The preimage of this point in $\cal{U}_w$, \emph{resp.}\ $\cal{X}_w$, is $\cal{U}_g$, \emph{resp.}\ $\cal{X}_g$.
Therefore, the maps above are the respective pullbacks to $\cal{U}_w$ and $\cal{X}_w$ of the isomorphism $G/H_g \xrightarrow{\sim} \cal{O}_w$ that sends $xH_g \mapsto (xB_+x^{-1}, xgB_+g^{-1}x^{-1})$.
\end{proof}

Note that when $\bb{F} = \bb{C}$, the maps on cohomology in \Cref{prop:g-to-w} preserve weight filtrations because the maps that induce them are algebraic.

\section{Khovanov--Rozansky Homology}\label{sec:kr}

\subsection{}

In this section, we prove \Cref{thm:cohomology} by way of more general constructions motivated by knot theory.

Let $\Br_W^+$ be the \emph{positive braid monoid} of $W$.
It is the monoid freely generated by a set of symbols $\{\sigma_w\}_{w \in W}$, modulo the relations $\sigma_{ww'} = \sigma_w\sigma_{w'}$ for all $w, w' \in W$ such that $\ell(ww') = \ell(w) + \ell(w')$.
The \emph{full twist} is the element $\pi = \sigma_{w_0}^2 \in \Br_W^+$.

For all $\beta = \sigma_{w_1} \cdots \sigma_{w_k} \in \Br_W^+$, we set
\begin{align}
\cal{U}(\beta) &= \{(u, B_1, \ldots, B_k) \in \cal{U} \times \cal{B}^k \mid u^{-1}B_ku \xrightarrow{w_1} B_1 \xrightarrow{w_2} \cdots \xrightarrow{w_k} B_k\},\\
\cal{X}(\beta) &= \{(B_1, \ldots, B_k) \in \cal{B}^k \mid B_k \xrightarrow{w_1} B_1 \xrightarrow{w_2} \cdots \xrightarrow{w_k} B_k\}.
\end{align}
Let $G$ act on these varieties by left conjugation.
We regard $\cal{U}(\beta)$ and $\cal{X}(\beta)$ as varieties over $\cal{O}_w$, where $w = w_1 \cdots w_k  \in W$, via the equivariant maps $(u, (B_i)_i) \mapsto (B_k, B_1)$ and $(B_i)_i \mapsto (B_k, B_1)$, respectively.
Deligne showed that up to isomorphism over $\cal{B} \times \cal{B}$, these varieties only depend on $\beta$, not on the sequence of elements $w_i$.
His full result describes the extent to which the isomorphism can be pinned down uniquely; see \cite{deligne} for details.

In particular, we have equivariant identifications
\begin{align}
\cal{U}_w &\xrightarrow{\sim} \cal{U}(\sigma_w),\\
\cal{X}_w &\xrightarrow{\sim} \cal{X}(\sigma_w \pi)
\end{align}
via $(u, B_1) = (u, B')$ and $(B_1, B_2, B_3) = (B', B'', B)$.

\subsection{}

If $W$ is the symmetric group on $n$ letters, denoted $S_n$, then the group completion of $\Br_W$ is the group of topological braids on $n$ strands, denoted $\Br_n$.
Any braid can be closed up end-to-end to form a link: that is, an embedding of a disjoint union of circles into $3$-dimensional space.
Thus there is a close relation between isotopy invariants of links and functions on the groups $\Br_n$.

In \cite{kr}, Khovanov and Rozansky introduced a link invariant valued in triply-graded vector spaces.
Its graded dimension can be written as a formal series in variables $a, q^{\frac{1}{2}}, t$.
In \cite{khovanov}, Khovanov showed how to construct it in terms of class functions on the groups $\Br_n$, and more precisely, in terms of functors on monoidal additive categories attached to the groups $S_n$.
When we set $t = -1$, the Khovanov--Rozansky invariant of a link specializes to its so-called HOMFLYPT series, and Khovanov's functors specialize to class functions originally introduced by Jones and Ocneanu.

The positive braid monoid $\Br_W^+$ and its group completion $\Br_W$ can actually be defined for any Coxeter group $W$, not just Weyl groups.
In \cite{gomi}, Y.\ Gomi extended the construction of Jones--Ocneanu to \emph{finite} Coxeter groups.
There is a similar extension of Khovanov's construction, up to a choice of a (faithful) representation on which $W$ acts as a reflection group.

Fix such a representation $V$.
For any braid $\beta \in \Br_W$, we write $\sf{HHH}_V(\beta)$ to denote the \emph{Khovanov--Rozansky (KR) homology} of $\beta$ with respect to $V$.
We will use the grading conventions in \cite{trinh}, so that
\begin{align}
\sf{P}_V(\beta) = 
(at)^{|\beta|} a^{-{\dim(V)}}
	\sum_{i,j,k} 
	{(a^2 q^{\frac{1}{2}}t)^{\dim(V) - i}} 
	q^{\frac{j}{2}}
	t^{-k} 
	\dim \sf{HHH}_V^{i,i+j,k}(\beta)
\end{align}
is an isotopy invariant of the link closure of $\beta$.
In the case where $W = S_n$, taking $V$ to be the $(n - 1)$-dimensional reflection representation yields what is usually called \emph{reduced KR homology} and denoted $\sf{HHH}$, while taking $V$ to be the $n$-dimensional permutation representation yields what is usually called \emph{unreduced KR homology} and denoted $\overline{\sf{HHH}}$.
They are related by
\begin{align}
\pa{\frac{a^{-1} + at}{q^{-\frac{1}{2}} - q^{\frac{1}{2}}}}
\sf{P}(\beta)
=	\overline{\sf{P}}(\beta),
\end{align}
where $\sf{P}$ and $\overline{\sf{P}}$ denote the series $\sf{P}_V$ for these respective choices of $V$.

Henceforth, let $r = \dim(V)$ and $N = \dim(\cal{B})$.
The results below are \cite[Cor.\ 4]{trinh} and \cite[Thm.\ 1.9]{ghmn}.

\begin{thm}\label{thm:cohomology-to-hhh}
Suppose that $W$ is the Weyl group of a split reductive group $G$ over $\bb{F}$ with root lattice $\Phi$, and that $V = \bb{Z}\Phi \otimes_{\bb{Z}} \bb{Q}$.
Then for any $\beta \in \Br_W^+$, we have isomorphisms
\begin{align}
\gr_{j + 2r}^\sf{W} \ur{H}_{c, G}^{j + k + 2r}(\cal{U}(\beta), \QL)
	&\simeq \sf{HHH}_V^{0, j, k}(\beta),\\
\gr_{j + 2(r + N)}^\sf{W} \ur{H}_{c, G}^{j + k + 2(r + N)}(\cal{X}(\beta), \QL)
	&\simeq \sf{HHH}_V^{r, r + j, k}(\beta)
\end{align}
for all $j, k$.
\end{thm}

\begin{thm}[Gorsky--Hogancamp--Mellit--Nakagane]\label{thm:ghmn}
For any integer $n \geq 1$ and $\beta \in \Br_n$, we have
\begin{align}
\overline{\sf{HHH}}^{0, j, k}(\beta)
	\simeq
\overline{\sf{HHH}}^{r, r + j, k}(\beta\pi)
\end{align}
for all $j, k$.
\end{thm}

\begin{proof}[Proof of {\Cref{thm:cohomology}}]
We must have $B_+ \xrightarrow{w} gB_+g^{-1}$ for some $w \in W$.
Combining \Cref{prop:v-to-x}, \Cref{prop:g-to-w}, and \Cref{thm:cohomology-to-hhh}, we get isomorphisms
\begin{align}
\gr_{j + 2n}^\sf{W} \ur{H}_{c, H_g}^{j + k + 2n}(\cal{U}_g, \QL) 
	&\simeq  \sf{HHH}_V^{0, j, k}(\sigma_w),\\
\gr_{j + 2(n + N)}^\sf{W} \ur{H}_{c, H_g}^{j + k + 2(n + N)}(\cal{X}_g, \QL) 
	&\simeq  \sf{HHH}_V^{r, r + j, k}(\sigma_w\pi),
\end{align}
where $V = \bb{Z}\Phi \otimes_{\bb{Z}} \bb{Q}$ and $\Phi$ is the root lattice of $G$.

If $G = \GL_n$, then $V$ is the permutation representation of $S_n$.
So in this case, $\sf{HHH}_V = \overline{\sf{HHH}}$, and we are done by \Cref{thm:ghmn}.
Finally, we bootstrap from $\GL_n$ to any other split reductive group of type $A$ using \Cref{lem:der,lem:ad}.
\end{proof}

\subsection{}

\Crefrange{thm:cohomology-to-hhh}{thm:ghmn} suggest the following generalization of \Cref{conj:main}.

\begin{conj}
For any $\beta \in \Br_W^+$, there is a homotopy equivalence between $\cal{U}(\beta)(\bb{C})$ and $\cal{X}(\beta\pi)(\bb{C})$ that matches the weight filtrations on their compactly-supported cohomology.
\end{conj}

\begin{rem}
It would be desirable to generalize the map of stacks $[\cal{V}_g/H_g] \to [\cal{U}_g/H_g]$ that arises from $\Phi_g$ to an explicit map $[\cal{X}(\beta\pi)/G] \to [\cal{U}(\beta)/G]$ for any positive braid $\beta$.
Due to the inexplicit nature of \Cref{lem:v-to-y}, we have not yet found such a generalization.
\end{rem}

\section{Point Counts over Finite Fields}\label{sec:ih}

\subsection{}

For any braid $\beta \in \Br_n$, we write $\hat{\beta}$ to denote its link closure.
The \emph{reduced HOMFLYPT series} $\bb{P}(\hat{\beta})$ is related to the KR homology of $\beta$ by
\begin{align}
\bb{P}(\hat{\beta}) = \sf{P}(\beta)|_{t \to -1}.
\end{align}
This is an element of $\bb{Z}[\![q^{\frac{1}{2}}]\!][q^{-\frac{1}{2}}][a^{\pm 1}]$.
We write $[a^i]\bb{P}(\hat{\beta})$ to denote the coefficient of $a^i$ in $\bb{P}(\hat{\beta})$, viewed as an element of $\bb{Z}[\![q^{\frac{1}{2}}]\!][q^{-\frac{1}{2}}]$.

If $\beta = \sigma_{s_1} \cdots \sigma_{s_\ell}$, where the elements $s_1, \ldots, s_\ell \in W$ are all simple reflections, then we set $|\beta| = \ell$.
This number only depends on $\beta$.
\Cref{thm:ghmn} then specializes to the following result from \cite{kalman}.

\begin{thm}[K\'alm\'an]
For any integer $n \geq 1$ and $\beta \in \Br_n$, we have
\begin{align}
{[a^{|\beta| - n + 1}]} \bb{P}(\hat{\beta})
=
{[a^{|\beta| + n - 1}]} \bb{P}(\widehat{\beta\pi}).
\end{align}
\end{thm}

In \cite[\S{8}]{trinh}, we generalized K\'alm\'an's result from $\Br_n$ to $\Br_W$.
In this section, we review the statement, then explain its relation to the point-counting identity \eqref{eq:kawanaka}.

\subsection{}

Let $H_W$ be the \emph{Iwahori--Hecke algebra} of $W$.
For our purposes, $H_W$ is the quotient of the group algebra $\bb{Z}[q^{\pm\frac{1}{2}}][\Br_W]$ by the two-sided ideal 
\begin{align}
\langle (\sigma_s - q^{\frac{1}{2}})(\sigma_s + q^{-\frac{1}{2}}) \mid \text{simple reflections $s$}\rangle.
\end{align}
For any element $\beta \in \Br_W$, we abuse notation by again writing $\beta$ to denote its image in $H_W$.

The sets $\{\sigma_w\}_{w \in W}$ and $\{\sigma_w^{-1}\}_{w \in W}$ are bases for $H_W$ as a free $\bb{Z}[q^{\pm\frac{1}{2}}]$-module.
Let $\tau^\pm : H_W \to \bb{Z}[q^{\pm\frac{1}{2}}]$ be the $\bb{Z}[q^{\pm\frac{1}{2}}]$-linear functions defined by:
\begin{align}
\tau^{\pm}(\sigma_w^{\pm 1}) &= \left\{\begin{array}{ll}
1	&w = e\\
0	&w \neq e
\end{array}\right.
\end{align}
For $W = S_n$, comparing $\tau^\pm$ with the Jones--Ocneanu trace on $H_W$ shows that
\begin{align}
[a^{|\beta| \pm (n - 1)}] \bb{P}(\hat{\beta})
	&= (q^{-\frac{1}{2}} - q^{\frac{1}{2}})^{-(n - 1)} 
		(-1)^{|\beta|} \tau^\pm(\beta)
\end{align}
for all $\beta \in \Br_n$.
Therefore the following result from \cite[\S{8}]{trinh} generalizes K\'alm\'an's theorem to arbitrary $W$.

\begin{thm}\label{thm:trace}
For any finite Coxeter group $W$ and braid $\beta \in \Br_W$, we have 
\begin{align}
\tau^-(\beta) = \tau^+(\beta\pi).
\end{align}
\end{thm}

\subsection{}

We return to the setting of \Cref{sec:kr}, so that $W$ is the Weyl group of $G$.
Under the hypotheses of \Cref{thm:cohomology-to-hhh}, the 
following identities from \emph{loc.\ cit.}\ relate \Cref{thm:trace} to point counting:
\begin{align}
\frac{|\cal{U}(\beta)(\bb{F})|}{|G(\bb{F})|}
	&=
		(q - 1)^{-r}
		(q^{\frac{1}{2}})^{|\beta|} \tau^-(\beta),\\[0.5ex]
\frac{|\cal{X}(\beta)(\bb{F})|}{|G(\bb{F})|}
	&=
		(q - 1)^{-r}
		(q^{\frac{1}{2}})^{|\beta|} \tau^+(\beta).
\end{align}
Together they imply:

\begin{cor}\label{cor:trace}
Keep the hypotheses of \Cref{thm:cohomology-to-hhh}.
Then for any $\beta \in \Br_W^+$, we have $|\cal{U}(\beta)(\bb{F})| = |\cal{X}(\beta \pi)(\bb{F})|$.
\end{cor}

We claim that when $B_+$ is defined over $\bb{F}$, \Cref{cor:trace} implies \eqref{eq:kawanaka} from the introduction.
The key is that the proof of \Cref{prop:g-to-w} also works at the level of $\bb{F}$-points.
Thus there are $G$-equivariant bijections
\begin{align}
(\cal{U}_g(\bb{F}) \times G(\bb{F}))/H_g(\bb{F}) \xrightarrow{\sim} \cal{U}_w(\bb{F}),\\
(\cal{X}_g(\bb{F}) \times G(\bb{F}))/H_g(\bb{F}) \xrightarrow{\sim} \cal{X}_w(\bb{F})
\end{align}
for any $g \in G(\bb{F})$ such that that $B_+ \xrightarrow{w} gB_+g^{-1}$.
Since the quotients are free, we deduce that
\begin{align}
|\cal{U}_g(\bb{F})||G(\bb{F})|
=	|\cal{U}_w(\bb{F})||H_g(\bb{F})|
=	|\cal{X}_w(\bb{F})||H_g(\bb{F})|
=	|\cal{X}_g(\bb{F})||G(\bb{F})|.
\end{align}
Applying \Cref{prop:v-to-x}, we arrive at $|\cal{U}_g(\bb{F})| = |\cal{X}_g(\bb{F})| = |\cal{V}_g(\bb{F})|$, which is \eqref{eq:kawanaka}.

\begin{rem}\label{rem:kawanaka}
The original identity proved by Kawanaka was
\begin{align}
|(\cal{U} \cap U_+\dot{w}B_+)(\bb{F})|
=	
|(U_+U_- \cap U_+\dot{w}B_+)(\bb{F})|
\end{align}
for all $w \in W$ \cite[Cor.\ 4.2]{kawanaka}.
This is equivalent to \eqref{eq:kawanaka} as long as $B_+$ is defined over $\bb{F}$.
For when the latter holds, an argument similar to the proof of \Cref{lem:translate} shows that the $\bb{F}$-point counts of $\cal{U}_g$ and $\cal{V}_g$ remain constant as $g$ runs over elements of $U_+\dot{w}B_+(\bb{F})$.
\end{rem}



\end{document}